\documentclass[a4paper]{amsart}

\usepackage{amssymb}
\usepackage{amsthm}  
\usepackage{amsmath} 
\usepackage{amscd} 
\usepackage[all]{xypic}
\usepackage{url}
\usepackage{color}

\newcommand{\om}{\omega}
\newcommand{\ka}{\kappa}

\newcommand{\ba}{\mathcal{G}}

\newcommand{\Ann}{\mathfrak{ann}}

\newcommand{\fk}{\mathfrak k}
\newcommand{\fh}{\mathfrak h}
\newcommand{\fm}{\mathfrak m}
\newcommand{\Ad}{{\rm Ad}}
\newcommand{\ad}{{\rm ad}}

\newcommand{\na}{\nabla}
\newcommand{\U}{\Upsilon}

\newcommand{\lz}{[\![}
\newcommand{\pz}{]\!]}

\newcommand{\R}{\mathbb{R}}

\newtheorem{prop*}{Proposition}

\newtheorem{thm*}{Theorem}

\newtheorem{lemma*}{Lemma}

\newtheorem{cor*}{Corollary}

\newtheorem{rem*}{Remark}
\theoremstyle{definition}
\newtheorem{def*}{Definition}

\theoremstyle{remark}
\newtheorem{exam}{Example}

\begin{document}
\title{Notes on symmetric conformal geometries}
\author{Jan Gregorovi\v c and Lenka Zalabov\' a}
\address{J.G. Department of Mathematics and Statistics, Faculty of Science, Masaryk University, Kotl\' a\v rsk\' a 2, Brno, 611 37, Czech Republic; L.Z. Institute of Mathematics and Biomathematics, Faculty of Science, University of South Bohemia, Brani\v sovsk\' a 1760, \v Cesk\' e Bud\v ejovice, 370 05, Czech Republic}
\email{jan.gregorovic@seznam.cz, lzalabova@gmail.com}
\thanks{First author supported by the grant P201/12/G028 of the Czech Science Foundation. Second author supported by the grant P201/11/P202 of the Czech Science Foundation.}

\maketitle
\begin{abstract} 
In this article, we summarize the results on symmetric conformal geometries. We review the results following from the general theory of symmetric parabolic geometries and prove several new results for symmetric conformal geometries. In particular, we show that each symmetric conformal geometry is either locally flat or covered by a pseudo-Riemannian symmetric space, where the covering is a conformal map. We construct examples of locally flat symmetric conformal geometries that are not pseudo-Riemannian symmetric spaces.
\end{abstract}

\section{Symmetric conformal geometries} \label{prvni-cast}
Assume $p+q \geq 3$ and $p \geq q \geq 0$ and let $M$  be a connected manifold of dimension $p+q$.
Two pseudo--Riemannian metrics $g$ and $\hat g$ of the signature $(p,q)$ on a manifold $M$ are known to be \emph{conformally equivalent}, if there is a positive smooth function $f$ such that $ \hat g=f^2g$. An equivalence class $[g]$ of pseudo--Riemannian metrics that are conformally equivalent to $g$ is then called a \emph{conformal structure} on $M$. A \emph{ (local) conformal transformation} is then a  (local) diffeomorphism of $M$ that preserves the conformal structure on $M$.

Each conformal structure on $M$ can be described as a first order G--structure corresponding to $CO(p,q) \subset GL(p+q,\R)$. The well known prolongation procedure allows to construct a parabolic geometry $(\ba \to M, \om)$ of type $(PO(p+1,q+1),P)$ from each such $CO(p,q)$--structure $M$, see \cite[Section 1.6.]{parabook}. Here $PO(p+1,q+1)$ is the projectivized pseudo--orthogonal group of signature $(p+1,q+1)$ and $P$ is the stabilizer of a null line in $\R^{p+q+2}$. 

In fact, there is a correspondence between conformal structures of signature $(p,q)$ on $M$ and parabolic geometries $(\ba \to M,\om)$ of type $(PO(p+1,q+1),P)$, which can be made bijective using the normalization condition on the curvature $\kappa$ of the parabolic geometry, see \cite[Section 3.1.]{parabook}. Moreover, for a normal parabolic geometry of type $(PO(p+1,q+1),P)$, the vanishing of $\kappa$  is equivalent to the vanishing of the Weyl tensor $W$ in the case $p+q>3$, and the vanishing of the Cotton-York tensor $C$ in the case $p+q=3$, of the corresponding conformal structure.
This leads to the description of conformal geometries as normal $|1|$--graded parabolic geometries, which is well understood, see \cite[Section 1.6.]{parabook}. We use this description to study conformal symmetries on conformal geometries.

\begin{def*}
A \emph{(local) conformal symmetry} at $x \in M$ on a conformal manifold $M$ is a (local) conformal transformation $S_x$ on $M$ such that:
\begin{enumerate}
\item $S_x(x)=x$,
\item $T_xS_x=-$id on $T_xM$.
\end{enumerate}
We call the geometry \emph{conformally symmetric}, if there exists a conformal symmetry at each point $x \in M$. We describe each \emph{system of conformal symmetries} on $M$ by a map $S: M \times M \to M$ given by $S(x,y)=S_x(y)$ for the chosen symmetries $S_x$ at $x$, and we call the system \emph{smooth}, if the map $S$ is smooth in both variables.
\end{def*}

The simplest example of a symmetric conformal geometry is the \emph{flat model} $PO(p+1,q+1)/P$, 
which is the M\" obius space of signature $(p,q)$.
It is a homogeneous space of dimension $p+q$ with the transitive action of $PO(p+1,q+1)$ and there is the natural  $PO(p+1,q+1)$--invariant conformal structure of signature $(p,q)$ induced from the standard metric in $\R^{p+q+2}$ of signature $(p+1,q+1)$. 

Let us follow the notation from \cite[Section 1.6.3]{parabook}. 
Thus we consider the standard metric in $\R^{p+q+2}$ given by the block matrix of the block form
$$
m=\left(
\begin{matrix}
0&0&1 \\ 0&J&0 \\ 1&0&0
\end{matrix}
\right)
$$
with blocks of sizes $1$, $p+q$ and $1$, where $J$ is the diagonal matrix with $p$ ones and $q$ minus ones on the diagonal. In particular, if $e_0, \dots, e_{p+q+1}$ is the canonical basis of $\R^{p+q+2}$, then $P$ is the stabilizer of the line $\langle e_0 \rangle$ generated by the first basis vector. Then
we write elements of $\frak{so}(p+1,q+1)$ as block matrices of the block form
$$
\left(
\begin{matrix}
a&Z&0 \\ X&A&-JZ^T \\ 0&-X^TJ&-a
\end{matrix}
\right)
$$
with blocks of sizes $1$, $p+q$ and $1$,
where $X \in \R^{p+q}$, $Z\in \R^{p+q*}$ and $(a,A)\in \frak{co}(p,q)$. 
One can verify that all conformal symmetries at the origin $eP$ are exactly left multiplications by elements of $PO(p+1,q+1)$ represented by matrices of the form
$$
s_{Z}=\left(
\begin{matrix}
-1&-Z& {1 \over 2}ZJZ^T \\ 0&E&-JZ^T \\ 0&0&-1
\end{matrix}
\right)
,$$ 
where $E$ is the identity matrix of the rank $p+q$, and $Z \in \R^{p+q*}$ can be arbitrary. In particular, all conformal symmetries $s_Z$ are involutive and there exists an infinite number of symmetries at each point $hP$ of $PO(p+1,q+1)/P$ given by matrices of the form $hs_Zh^{-1}$ for all $Z \in \R^{p+q*}$.

The conformal geometry is called \emph{(conformally) locally flat}, if $p+q>3$ and $W=0$, or $p+q=3$ and $C=0$. In particular, it is locally flat if and only if it is locally isomorphic to the flat model.

 We will see that conformal symmetries on conformal geometries such that $W\neq 0$ have very different properties than the symmetries on flat models.

\section{Distinguished connections}
Suppose we have a conformal structure $(M,[g])$ and suppose there is a conformal symmetry $S_x$ at $x$. It is well known that there always is a class of torsion--free affine connections preserving the conformal structure $[g]$, the so--called \emph{Weyl connections}, see \cite[Section 1.6.]{parabook}. Arbitrary two Weyl connections $\na$ and $\hat \na$ from the class are related by 
\begin{align} \label{rozdil}
\hat \na_\xi(\eta)=\na_\xi(\eta)+
\U(\xi)\eta+\U(\eta)\xi-g(\xi, \eta)g^{-1}(\U)
\end{align} 
for suitable $\U \in \Omega^1(M)$, where $g$ is an 
arbitrary metric from the conformal class. 
In particular, if $S_x$ is a conformal symmetry and $\na$ is a Weyl connection, then $S_x^*\na$ is a Weyl connection, too. In other words, there exists $\U \in \Omega^1(M)$ such that 
\begin{align} \label{zmena}
(S_x^*\na)_\xi(\eta)=\na_\xi(\eta)+
\U(\xi)\eta+\U(\eta)\xi-g(\xi, \eta)g^{-1}(\U),
\end{align}
and this holds locally in the case of local symmetries.

In fact, the linear map $\R^{p+q}\to \R^{p+q}$ given by
\begin{align} \label{alg}
(\U(\xi)(-)+\U(-)(\xi)-g(\xi,-)g^{-1}(\U))(x)
\end{align} 
naturally provides an element  $\xi(\U)$ of $\frak{co}(p,q)$ with the same action on $\R^{p+q}$ for each $x$. In particular, the evaluation of the formula in $\xi$ is a linear map $\xi(-): \R^{p+q*}\to \frak{co}(p,q)$ at each $x$. We recall that $\xi(\U)=0$ for all $\xi\in \R^{p+q}$ if and only if $\U=0$.

Suppose there is a (local) conformal symmetry $S_x$ at $x$ on $M$. It is proved in \cite{ja-elsevier} that there always is a Weyl connection $\na^{S_x}$ satisfying $S_x^*\na^{S_x}=\na^{S_x}$. The `difference' $\U$ between $\na^{S_x}$ and any other Weyl connection $\nabla$ satisfying $S_x^*\na=\na$ given by the formula (\ref{rozdil}) vanishes at $x$, because $\U(x)$ is $S_x$--invariant tensor of odd type and thus vanishes at $x$. Moreover, this implies for the Weyl tensor $W$ that $\na W$ vanishes at $x$ for any Weyl connection satisfying $S_x^*\na=\na$, because $(\na W)(x)$ is a tensor of type odd invariant with respect to $S_x$.

Let us remark that in the case $p+q=3$, the Weyl tensor vanishes identically and the Cotton--York tensor $C$ vanishes at the points with conformal symmetries, since it is an invariant tensor of odd. In particular, if the geometry is conformally symmetric, then the geometry is locally flat.

Assume there is a smooth system of conformal symmetries $S$ on the conformal manifold $M$. For each such system $S$, there can exist at most one connection $\na^S$ which is invariant with respect to $S$, i.e. such that $S_x^*\na^S=\na^S$ for all $x \in M$. To get this uniquely given candidate for the invariant connection $\na^S$, one simply glues pointwise the above connections $\na^{S_x}$ for all $x$, see \cite{ja-springer} for details. The connection $\na^S$ always satisfies $\na^S W =0$ on $M$. However, $\na^S$ need not to be invariant and it is proved in \cite{ja-springer} that this is the case if and only if $S_x \circ S_y \circ S_x=S_{S_x(y)}$ holds for each $x, y \in M$. In such situation, $(M,S)$ is an symmetric space.

Let us now focus on the question, when the condition $S_x \circ S_y \circ S_x=S_{S_x(y)}$ is satisfied? There are some partial results in \cite{ja-elsevier}. We present here a particular improvement of the results, which includes the answer to our question.

\begin{prop*} \label{dve-symetrie}
Suppose there are two different symmetries at $x$ on the conformal structure $(M,[g])$. Then the Weyl tensor $W$ vanishes at $x$.
\end{prop*}
\begin{proof}
We follow here \cite[Theorem 4.5.]{ja-elsevier}.
Consider two different symmetries $S_x$ and $S'_x$ at $x$.
Then there are two different Weyl connections such that $(S_x)^*\na^{S_x}=\na^{S_x}$ and $(S_x')^*\na^{S_x'}=\na^{S_x'}$ hold on a neighborhood of $x$, and thus $\na^{S_x} W=0$ and $\na^{S_x'} W=0$ at $x$. Moreover, there is $\U\neq 0$ given by (\ref{rozdil}) that relates directly the two derivatives $\na^{S_x}$ and $\na^{S_x'}$. If we apply this on the Weyl tensor $W$, we get that the element $\xi(\U)$ of $\frak{co}(p,q)$ given by (\ref{alg})
has to act trivially on $W(x)$ for all vectors $\xi$.  
We show in the following Lemma that this always implies vanishing of $W(x)$.
\end{proof}

Let us denote by $\Ann(W_x)$ the set of all $A\in \frak{co}(p,q)$ such that $A$ acts trivially on $W(x)$. We need to show that the set  
$$\Ann(W_x)^{(1)}=\{\U : \xi(\U)\in \Ann(W_x)\ \rm{for\ all} \ \xi\in \R^{p+q}\}$$ is trivial if $\Ann(W_x)\neq \frak{co}(p,q)$.

\begin{lemma*} \label{ann-trivial}
If $W(x)$ is non--trivial, then $\Ann(W_x)^{(1)}=0$.
\end{lemma*}
\begin{proof}
The statement follows from the general theory of \cite{KT}. 
One can also use the ideas of \cite[Section 3.2]{CM}, where the authors compute directly the actions of suitable elements $\xi(\U)\in \mathfrak{co}(p,q)$ to get vanishing of the Weyl tensor as a consequence of the existence of automorphisms with higher order fixed points, which in our case correspond to the elements $\U$.
\end{proof}

There is the following direct consequence of the Proposition \ref{dve-symetrie}.
\begin{cor*} \label{existence}
Let $(M,[g])$ be a conformal structure and assume $W(x)\neq 0$.
\begin{enumerate}
\item There can exist at most one conformal symmetry at $x$ on $M$.
\item If there is a conformal symmetry $S_z$ at $z$ and $S_z(y)=x$, then the condition  $S_z \circ S_y \circ S_z=S_{S_z(y)}$ is trivially satisfied for $z$ and arbitrary conformal symmetry $S_y$ at $y$.
\end{enumerate}
\end{cor*}

Moreover, if $W(x)\neq 0$, then $W\neq 0$ on a neighborhood of $x$ and if there are conformal symmetries, then they are unique, too. 
We can prove that the corresponding system of conformal symmetries is smooth.

\begin{lemma*} \label{hladkost-system}
Let $(M,[g])$ be a conformal structure, let $U$ be an open subset of $M$ with non--trivial $W$ and suppose there is a conformal symmetry at each $x\in U$. Then the system $S$ assigning to $x$ the unique symmetry at $x$ is smooth map from $U$ to the automorphism group.
\end{lemma*}
\begin{proof}
Let us fix a Weyl connection $\nabla^{S_x}$ such that $(S_x)^*\nabla^{S_x}=\nabla^{S_x}$ for some $x\in U$. Then for each $y\in U$, there is $\U(y)$ describing the `difference' between $\nabla^{S_y}$ and $\nabla^{S_x}$ at $y$ by (\ref{rozdil}). Thus $\nabla^{S_x} W(y)$ is given by the algebraic action (\ref{alg}) of $\U(y)$ on $W$ for each $\xi\in \R^{p+q}$. To prove the claim, it is enough to show that $\U(y)$ depends smoothly on $y$. But $\nabla^{S_x} W(y)$ is smooth and thus the image of $\xi(\U(y))\in \frak{co}(p,q)$ depends smoothly on $y$ for each $\xi\in \R^{p+q}$. Since the kernel of the action coincides with $\Ann(W_y)^{(1)}$, we conclude that $\U(y)$ depends smoothly on $y$. 
\end{proof}

The Lemma particularly implies, that if $U$ consists of
all points with non--trivial $W$ and there is a conformal symmetry at each $x\in U$, then $U$ is closed in $M$, because the group generated by symmetries at points in $U$ acts transitively on $U$. Since $U$ is open and $M$ connected, it holds $U=M$. We know from Corollary \ref{existence} that conformal symmetries on $U$ satisfy the condition $S_x \circ S_y \circ S_x=S_{S_x(y)}$ and we can summarize the above results in the following Theorem.

\begin{thm*}
Suppose $(M,[g])$ is a conformally symmetric conformal structure with the Weyl tensor $W$. Then one of the following claims hold:

\begin{enumerate}
\item $W=0$,
\item $W\neq 0$ and $(M,S)$ is a symmetric space.
\end{enumerate}

In particular, each conformally symmetric conformal structure is either conformally homogeneous or locally flat.
\end{thm*}

Since each non--flat conformal symmetric geometry is conformally homogeneous, the group $K$ of conformal transformations generated by the corresponding smooth system of symmetries $S$ acts transitively on $M$, and we get $M=K/H$, where $H \subset K$ is the stabilizer of a point. As usual, there is the canonical $\Ad_H$--invariant decomposition $\fk=\fh \oplus \mathfrak{m}$, and $s \in H$ such that $\fh$ is $1$--eigenspace of $\Ad_s$, and $\fm$ is $-1$--eigenspace of $\Ad_s$. In other words, we have the  homogeneous symmetric space $(K,H,s)=(M,S)$ for the group $K$ generated by all symmetries.


\section{Symmetries of homogeneous conformal geometries}
Let $(\ba \to M, \om)$ be a normal homogeneous conformal geometry. For an arbitrary transitive group $K$ of conformal transformations and a stabilizer $H \subset K$, the choice of $u \in \ba$ from the fiber over $eH \in K/H \simeq M$ provides an inclusion $j: K \rightarrow \ba$ of the form $j(k)=k(u)$, and Lie group homomorphism $i:H \rightarrow P$ of the form $h(u)=u\cdot i(h)$. Then $\omega(u)\circ T_{e}j$  defines a linear map $\alpha: \mathfrak{k}\to \mathfrak{\frak{so}}(p+1,q+1)$ satisfying the following conditions:
\begin{enumerate}
\item $\alpha: \fk \to \mathfrak{so}(p+1,q+1)$ is linear map extending $T_ei:\mathfrak{h}\to \mathfrak{p}$,
\item $\alpha$ induces an isomorphism $\alpha: \mathfrak{k}/\mathfrak{h} \rightarrow \mathfrak{so}(p+1,q+1)/\mathfrak{p}$ of vector spaces, 
\item $\Ad(i(h))\circ \alpha=\alpha \circ \Ad(h)$ holds for all $h\in H$.
\end{enumerate}

Then there is $P$--bundle isomorphism $\mathcal{G} \simeq K\times_{i(H)} P$ such that $u=\lz e,e \pz$, and the Cartan connection $\om$ can be described as $\omega(j(k))(X,A)=\alpha \circ \omega_K(k)(X)+\omega_P(e)(A)$ for $X \in \frak{so}(p+1,q+1)$ and $A \in \frak{p}$, where $\om_G$ denotes the Maurer--Cartan form on the Lie group $G$.
Let us recall that such pair of maps $(\alpha,i)$ is called an \emph{extension} of $(K,H)$ to $(PO(p+1,q+1),P)$, see \cite[Section 1.5.15]{parabook}. The curvature $\kappa$ is completely determined by $\kappa(u)(\alpha(X),\alpha(Y))=[\alpha(X),\alpha(Y)]-\alpha([X,Y])$ for $X,Y\in \fk/\fh$.

Let us recall results on (local) existence of symmetries on a homogeneous conformal geometry following from \cite[Section 2]{GZ-Lie}.

\begin{prop*}\label{existence-2}
Let  $(\alpha,i)$ be an extension of $(K,H)$ to $(PO(p+1,q+1),P)$. 
\begin{enumerate}
\item If $\kappa=0$, then there is locally defined symmetry at each point.
\item If $\kappa\neq 0$, then there is a (unique) local symmetry at each point preserving $\fk$ if and only if there is $Z\in \R^{p+q*}$ such that $\Ad_{\exp Z} \alpha(\fk)$ is preserved by $\Ad_{s_0}$, 
where $s_0$ is the symmetry of the flat model as described in the Section \ref{prvni-cast} for $Z=0$.

In particular, if $\fk$ coincides with the algebra of all complete conformal Killing fields, then there is a (unique) local symmetry at each point if and only if there is $Z\in \R^{p+q*}$ such that $\Ad_{\exp Z} \alpha(\fk)$ is preserved by $\Ad_{s_0}$.
\end{enumerate}
\end{prop*}

Suppose that $(\ba \to M, \om)$ is a $K$--homogeneous conformal geometry on homogeneous symmetric space $(K,H,s)$, where $K$ is the group generated by all symmetries. Then it is proved in \cite[Proposition 2.3]{HG-CEJM} that the geometry can be described by the extension $(\alpha,i)$ of $(K,H)$ to $(PO(p+1,q+1),P)$ which particularly satisfies $i(H) \subset CO(p,q)$ and $i(s)=s_0$.
Precisely, it is proved in \cite[Proposition 2.5]{HG-CEJM} that there is a bijection between:
\begin{itemize}
\item extensions of $(K,H)$ to  a conformal structure $(PO(p+1,q+1),P)$ such that $i(s) = s_0$, and  
\item couples $(\beta, b)$, where $\beta$ is a frame of $\fm$ such that the inclusion $i_\beta: H \to GL(\R^{p+q})$ induced by the frame $\beta$ is contained in $CO(p,q)$, and $b$ is an endomorphism of $\R^{p+q*}$ commuting with $i_\beta(\Ad(H))$. 
\end{itemize}
In fact, $i_\beta$ is $\Ad:H \to GL(\fm)$ written in the frame $\beta$ and the $b$--part plays no important role for now.

Such a pair $(\alpha,i)$ determines a reduction with respect to $CO(p,q) \to P$ of the homogeneous geometry $(\ba \to M,\om)$ to $K$--invariant $CO(p,q)$--structure over $M$. The reduction can be viewed as a $(K,CO(p,q))$--orbit of a suitable $u \in \ba$.

Moreover, there is a reduction with respect to $i(H) \to P$ of the homogeneous geometry $(\ba \to M,\om)$ to $K$--invariant $i(H)$--structure over $M$, which can be viewed as $(K,i(H))$--orbit of a suitable $u$. It follows from the reductivity of the pair $(K,H)$ that this reduction corresponds to the canonical invariant connection on the symmetric space. In the language of the bundle isomorphisms, we get $K\times_{i(H)} i(H) \subset   K\times_{i(H)} CO(p,q) \subset  K\times_{i(H)}P \simeq \mathcal{G}$. 

Clearly, the connection is the Levi--Civita connection of a metric from the conformal class if and only if $i(H) \subset O(p,q)$, i.e. if the above reduction is a reduction to $O(p,q)$--structure on $M$.

\begin{thm*}
Let $(\alpha,i)$ be the extension of $(K,H)$ to $(PO(p+1,q+1),P)$ for the homogeneous symmetric space $(K,H,s)$. Then there is $K$--invariant metric $g$ of signature $(p,q)$ on $K/H^0$, where $H^0$ is the component of identity of $H$.

In particular, each conformally symmetric conformal structure such that $W\neq 0$ is covered by a pseudo--Riemannian symmetric space, where the coving is a conformal map.
\end{thm*}
\begin{proof}
Let $(\alpha,i)$ be the extension of $(K,H)$ to $(PO(p+1,q+1),P)$ described above. We need to show that $i(H^0) \subset O(p,q)$, because $(\alpha,i)$ naturally restricts to a extension of $(K,H^0)$ to $(G,P)$. This implies that the corresponding connection on $K/H^0$ is metric, and then the symmetric space $(K,H^0,s)$ is pseudo--Riemannian and the projection $K/H^0\to K/H$ is clearly a conformal map.

%
Consider the canonical decomposition $\fk = \fh \oplus \fm$. Then $[\fm,\fm] \subset \fh$ holds, and thus $\exp([X,Y]) \in H^0$ holds for each $X,Y \in \fm$. Then the elements
$$i(\exp([X,Y]))=\exp(i'([X,Y]))$$ for all $X,Y \in \fm$ 
generate $H^0$. Thus $i=\Ad$ for the right choice of the base, and so we will study elements of the form $\exp(\ad([X,Y]))$ for all $X,Y \in \fm$. 

The element $\exp(\ad([X,Y]))$ belongs to $\frak{so}(p,q)$ if and only if its determinant equals to $1$, i.e. if and only if the trace of $\ad([X,Y])$ equals to $0$. But we have 
$\ad([X,Y])=\ad_X \circ \ad_Y - \ad_Y \circ \ad_X$ and the trace equals to
$$
tr(\ad([X,Y]))=tr(\ad_X \circ \ad_Y - \ad_Y \circ \ad_X)=B(X,Y)-B(Y,X)
,$$
where $B$ denotes the Killing from, which is symmetric. Then $i(H^0) \subset O(p,q)$ and the statement follows.
\end{proof}

\section{Locally flat conformal symmetric spaces}
On the locally flat geometries, there can exist many systems of conformal symmetries, because there can be many conformal symmetries at one point.
However, there are symmetric locally flat geometries that cannot form pseudo--Riemannian symmetric spaces.
We follow \cite{ja-CEJM} to construct such an example, where any system of conformal symmetries is not smooth.

\begin{exam}
Let us start with the flat model $PO(p+1,q+1)/P$, where $P$ is the stabilizer of the line $\langle e_0 \rangle$ generated by the first basis vector $e_0$ of the canonical basis. Assume $p,q>0$ and consider the locally flat manifold $M:=PO(p+1,q+1)/P - \{\langle u \rangle,\langle v\rangle\}$, where $u,v \in \R^{p+q+2}$ are arbitrary non--zero null vectors isotropic for $m$, i.e. $m(u,v)=0$. Then $M$ clearly is a conformal manifold and its group of conformal transformations $K(u,v)$ consists of those elements of $PO(p+1,q+1)$ that preserve $\langle u \rangle,\langle v\rangle$ or that swap them. Clearly, $K(u,v)$ has two connected components and the connected component of identity is the intersection of two parabolic subgroups of $PO(p+1,q+1)$ first of them preserving $\langle u \rangle$ and second of them preserving $\langle v \rangle$.

The group $K(u,v)$ does not act transitively on $M$, because the parabolic subgroup preserving arbitrary point $\langle w \rangle$ has three orbits in $PO(p+1,q+1)/P$ characterized by the isotropy and linear dependence with respect to $w$. Thus each orbit of $K(u,v)$ is determined according to the isotropy of $u, v$ and according to the relative position to the line $\langle u,v\rangle$.

We show that there exist symmetries at points of each orbit of $K(u,v)$. Instead of fixing $\langle u\rangle,\langle v \rangle$ and discussing symmetries at various points $\langle w \rangle$, we fix the point $\langle e_0 \rangle$ and we choose admissible $\langle u \rangle$ and $\langle v \rangle$ such that $\langle e_0 \rangle$ lies in the right orbit. Then we find a symmetry at $\langle e_0 \rangle$, and there are symmetries of the same behavior on the whole orbit given by $\langle e_0 \rangle$ due to transitivity. Let us remind the description of all symmetries at the origin $\langle e_0 \rangle$:
$$
s_{Z}=\left(
\begin{matrix}
-1&-Z& {1 \over 2}ZJZ^T \\ 0&E&-JZ^T \\ 0&0&-1
\end{matrix}
\right)
,$$ 
where $Z=(z_1, \dots, z_{p+q}) \in \R^{p+q*}$ is arbitrary.

Let us start with the orbit corresponding to the situation $m(e_0,u)\neq 0$ and $m(e_0,v)\neq 0$. We can choose $u=e_0+\sqrt{2}e_1-e_{p+q+1}$ and $v=e_0-\sqrt{2}e_{p+q}+e_{p+q+1}$. There is exactly one symmetry $s_Z$ for $Z$ of the form $z_1=-\sqrt{2}$ and $z_{p+q}=\sqrt{2}$ and $z_i=0$ for the remaining indexes. This symmetry swaps $\langle u \rangle$ and $\langle v \rangle$ and there is no symmetry preserving them.

Let us now consider the orbit for $m(e_0,u)=0$ and $m(e_0,v)\neq 0$ (which is the same orbit as the situation $m(e_0,u)\neq 0$ and $m(e_0,v)=0$). We choose $u=e_1+e_{p+q}$ and $v=e_{p+q+1}$. If there is a symmetry, then it has to preserve $\langle u \rangle$ and $\langle v \rangle$ due to the isotropy, and it turns out that there is exactly one symmetry $s_Z$ for $Z=0$.

The next possibility is the orbit for the situation $m(e_0,u)=m(e_0,v)=0$ and $e_0 \in \langle u,v \rangle$. We can choose $u=e_{1}+e_{p+q}$
 and $v=e_0+e_{1}+e_{p+q}$, and there are  (generally many) symmetries $s_Z$ for $Z$ satisfying $z_1+z_{p+q}+1=0$, which swap $\langle u \rangle$ and $\langle v \rangle$.
There are no symmetries preserving them.

In fact, this covers all possible orbits  for the case $p=1$ or $q=1$, i.e. the Lorentzian signature. In other cases, there is the remaining orbit for the case $m(u,e_0)=m(v,e_0)=0$ and $e_0 \notin \langle u,v \rangle$. We can choose $u=e_1+e_{p+q}$ and $v=e_2+e_{p+q-1}$. There are (generally many) symmetries $s_Z$ for $Z$ satisfying $z_1+z_{p+q}=0$ and $z_2+z_{p+q-1}=0$. These symmetries preserve $\langle u \rangle$ and $\langle v \rangle$ and there are no symmetries swapping them.
 
Altogether, there is no smooth system of symmetries, although there can exist infinite amount of symmetries at the points. 
In particular, it cannot form a pseudo--Riemannian symmetric space.
\end{exam}

Let us point out that this principle does not work in the case $p,q>0$, if we remove two points corresponding to non--isotropic vectors. This also gives a counter--example for the existence of global symmetries in the case $\ka=0$, see  Proposition \ref{existence-2}. In fact, the global existence of symmetries is the question of topology of the manifold.

\begin{exam}
Assume $p,q>0$ and consider the locally flat manifold $M:=PO(p+1,q+1)/P - \{\langle u \rangle,\langle v\rangle\}$, where $u,v \in \R^{p+q+2}$ are arbitrary non--zero null vectors non--isotropic for $m$, i.e. $m(u,v) \neq 0$. Choose $u=e_{p+q+1}$ and $v=e_0+e_1+e_{p+q}$, thus $m(e_0,u) \neq 0$, $m(e_0,v)=0$ and $m(u,v) \neq 0$. 
There cannot exist a symmetry at $\langle e_0 \rangle$ swapping $\langle u \rangle$ and $\langle v\rangle$ due to isotropy. It turns out that there is also no symmetry preserving them. 
Indeed, one can see from the previous example that each symmetry preserving $\langle u \rangle$ 
is of the form $s_Z$ for $Z=0$, while each symmetry $s_Z$ preserving 
$\langle v \rangle$ has to satisfy $z_1+z_{p+q}+2=0$.
\end{exam}

Let us finally show that the case $p=0$ or $q=0$, i.e. the Riemannian signature, behaves differently. Clearly, we cannot remove points corresponding to two isotropic vectors, because there are no two such vectors available. However, we can remove points $\langle u \rangle$ and $\langle v \rangle$ corresponding to two non--isotropic vectors $u,v$. Consider $q=0$. The manifold $M:=PO(1,p+1)/P -\{\langle u \rangle,\langle v \rangle\}$ is still homogeneous for the subgroup of $PO(1,p+1)$ consisting of elements preserving or swapping $\langle u \rangle$ and $\langle v \rangle$, and $M$ consists of points $\langle w \rangle$ for vectors $w$ satisfying $m(w,u) \neq 0$ and $m(w,v) \neq 0$. Let us choose $u=-e_0+\sqrt{2}e_{p}+e_{p+1}$ and $v=e_0+\sqrt{2}e_{p}-e_{p+1}$, and find the symmetry at $\langle e_0 \rangle$. It turns out that there is exactly one symmetry $s_Z$ for $Z=0$ and this symmetry swaps $\langle u \rangle$ and $\langle v \rangle$. There is no symmetry preserving $\langle u \rangle$ and $\langle v \rangle$. Altogether, $M$ is a Riemmanian symmetric space with a smooth system of symmetries given by the above symmetry $s_0$. In fact, it is just cylinder $M\cong \mathbb{R}\times SO(p)/SO(p-1)$.

\section{Remark on manifolds with parallel Weyl tensor}
There is a different definition of conformally symmetric spaces available. In \cite{polaci}, the authors define a \emph{conformally symmetric manifold} as a manifold $M$ together with a pseudo--Riemannian metric $g$ such that the Weyl tensor is parallel for the Levi--Civita connection of $g$. They call the conformally symmetric manifold  $(M,g)$ \emph{essential}, if it is not flat and it is not a locally symmetric space. There is a natural question on relations between this concept and our concept of conformal symmetries and conformal symmetric geometries. However, it turns out that this relation is not significant.

Let $(M,g)$ be a pseudo--Riemannian manifold with a non--zero parallel Weyl tensor for the Levi-Civita connection $\na$ of $g$. Suppose there is a conformal symmetry $S_x$ at $x$. Then there is a connection $\hat \na$ which is invariant with respect to $S_x$ on a neighborhood of $x$, and in particular, $\hat \na W$ vanishes at $x$.  It follows directly from the proof of the Proposition \ref{dve-symetrie} that the existence of such $\hat \na$ implies vanishing of $W$ at $x$. 
Thus if there is a conformal symmetry $S_x$ at each point $x$, then the unique corresponding system has to have the Levi--Civita connection $\na$ as its invariant connection, and $(M,g)$ is a pseudo--Riemmanian symmetric space.

Altogether, essential conformally symmetric spaces in the sense of \cite{polaci} cannot carry a conformal symmetry at each point. The same holds for each open subset. However, we know nothing about the existence of an isolated symmetry at some point, yet.

\end{document}